\documentclass{amcjou}

\def \r{\mathbb R}

\DeclareMathOperator{\sgn}{sgn}
\DeclareMathOperator{\arccotan}{arccotan}

\begin{document}
\input{epsf}

\begin{frontmatter}   %%  Title, information about author, abstract, etc.

\titledata{On stratifications for planar tensegrities with a small number of vertices}           % title of the paper
{}                                              % footnote on the title -- empty if not required

\authordata{Oleg Karpenkov}                                 % First author name
{TU Graz, Kopernikusgasse 24, A8010, Graz, Austria}         % Affiliation and address
{karpenkov@tugraz.at}                                       % E-mail address
{Partially supported by RFBR SS-709.2008.1 grant and
by FWF grant No.~M~1273-N18.}

\authordata{Jan Schepers}            % Second author
{Departement Wiskunde, Katholieke Universiteit Leuven,
Celestijnenlaan 200B, 3001, Leuven, Belgium}      %
{janschepers1@gmail.com}
{Postdoctoral Fellow of the Research Foundation - Flanders (FWO).}                                       % No footnote!

\authordata{Brigitte Servatius}            % Second author
{Mathematics Department, Worcester Polytechnic Institute,
100 Institute Road, Worcester, MA 01609-2280, USA}      %
{bservat@math.wpi.edu}
{}                                       % No footnote!

               % Keywords
                                             % Math. Subj. Class. codes

\begin{abstract}
In this paper we discuss several results about the structure of the
configuration space of two-dimensional tensegrities with a small
number of points. We briefly describe the technique of surgeries that
is used to find geometric conditions for tensegrities. Further we
introduce a new surgery for three-dimensional tensegrities. Within
this paper we formulate additional open problems related to the
stratification space of tensegrities.
\end{abstract}

\end{frontmatter}   %

\section{Introduction}

In this paper we study the stratification spaces of tensegrities
with a small number of points. We work mostly with planar
tensegrities. In the case of 4 and 5 point configurations we
give an explicit description of all the strata and present a
visualization of the entire stratification space. Further we give
a geometric description of the strata for 6 and 7 points and use the
technique of surgeries to find new geometric conditions adding to the
list of already known ones. In particular, we introduce a new
surgery for tensegrities in $\r^3$.

\subsection{Configuration space of tensegrities}

The first steps in the study of rigidity and flexibility of
tensegrities were made by B.~Roth and W.~Whiteley in~\cite{Rot1}
and further developed by R.~Connelly and W.~Whiteley
in~\cite{Con}, see also the survey about rigidity in~\cite{Whi}.
N.~L.~White and W.~Whiteley in~\cite{WW} started the
investigation of geometric conditions for a tensegrity with
prescribed bars and cables. In the preprint~\cite{Guz2}
M.~de~Guzm\'an describes several other examples of geometric
conditions for tensegrities.

Let us recall standard definitions of tensegrities (as
in~\cite{CC}, \cite{DKS}, etc.). See also~\cite{Ser} for a
collection of open problems and a good bibliography.

\begin{defn}
Fix a positive integer $d$. Let $G=(V,E)$ be an arbitrary graph
without loops and multiple edges. Let it have $n$ vertices $v_1,\ldots , v_n$.
\itemize
\item{A {\it configuration} is a finite collection $P$ of $n$
labeled points $(p_1,p_2,\ldots,p_n)$, where each point $p_i$
(also called a {\it vertex}) is in a fixed Euclidean space
$\r^d$.}
\item{The embedding of $G$ with straight edges, induced by mapping $v_j$ to $p_j$ is called a {\it tensegrity framework} and it is denoted as $G(P)$.}
\item{We say that a {\it load} or {\it force} $F$ acting on a
framework $G(P)$ in $\r^d$ is an assignment of a vector $f_i$ in
$\r^d$ to each vertex $i$ of $G$.}
\item{We say that a {\it stress} $w$ for a framework $G(P)$ in
$\r^d$ is an assignment of a real number $w_{i,j}=w_{j,i}$ (we
call it an {\it edge stress}) to each edge $p_ip_j$ of $G$. An
edge stress is regarded as a tension or a compression in the edge
$p_ip_j$. For simplicity reasons we put $w_{i,j}=0$ if there is no
edge between the corresponding vertices. We say that $w$ {\it
resolves} a load $F$ if the following vector equation holds for
each vertex $i$ of $G$:
$$
f_i+\sum\limits_{\{j|j\ne i\}} w_{i,j}(p_j-p_i)=0.
$$
By $p_j{-}p_i$ we denote the vector from the point $p_i$ to the
point $p_j$.}
\item{A stress $w$ is called a {\it self stress} if, the
following equilibrium condition is fulfilled at every vertex
$p_i$:
$$
\sum\limits_{\{j|j\ne i\}} w_{i,j}(p_j-p_i)=0.
$$
}
\item{A couple $(G(P),w)$ is called a {\it tensegrity} if
$w$ is a self stress for the framework $G(P)$.}

\item If $w_{i,j} < 0$ then we call the edge $p_ip_j$ a \textit{cable}, if $w_{i,j} > 0$ we call it a \textit{strut}.
\end{defn}

Let $W(n)$ denote the linear space of dimension $n^2$ of all
edge stresses $w_{i,j}$. Consider a framework $G(P)$ and denote
by $W(G,P)$ the subset of $W(n)$ of all possible self stresses
for $G(P)$. By definition the set $W(G,P)$ is a linear subspace
of $W(n)$.

\begin{defn}
The {\it configuration space of tensegrities} corresponding to the
graph $G$ is the set
$$
\Omega_d(G):=\big\{(G(P), w)\,|\,P\in (\r^d)^n, w \in
W(G,P)\big\}.
$$
The set $\{G(P)\,|\,P\in (\r^d)^n\}$ is said to be the {\it base
of the configuration space}, we denote it by $B_d(G)$.
\end{defn}

\subsection{Stratification of the base of a configuration space of tensegrities}

Suppose we have some framework $G(P)$ and we want to find a
cable-strut construction on it. Then {\it which edges can be
replaced by cables, and which by struts?} {\it What is the
geometric position of points for which given edges may be replaced
by cables and the others by struts?} These questions lead to the
following definition.

\begin{defn}
A linear fiber $W(G,P_1)$ is said to be {\it equivalent} to a
linear fiber $W(G,P_2)$ if there exists a homeomorphism $\xi$
between $W(G,P_1)$ and $W(G,P_2)$, such that for any self stress
$w$ in $W(G,P_1)$ the self stress $\xi(w)$ satisfies
$$
\sgn\big(\xi(w)\big)=\sgn\big(w\big).
$$
\end{defn}

The described equivalence relation gives us a stratification of
the base $B_d(G)=(\r^d)^n$. A {\it stratum} is by definition a
maximal connected set of points with equivalent linear fibers. In the
paper~\cite{DKS} we prove that all strata are semialgebraic sets
(which implies for instance that they are path connected).

The idea of this paper is to make the first steps in the study of
particular configuration spaces of tensegrities. We present the
techniques to find geometric conditions and open problems
for further study that already arise in very simple situations of
9 point configurations.

Let us, first, make the following three {\it general remarks}.

\vspace{1mm}

{\bf GR1.} The majority of the strata of codimension $k$ can be defined
by algebraic equations and inequalities that define the strata
of codimension~1. The exceptions here are mostly in high
codimension (the simplest one is as follows: for two points connected by an edge there is no codimension 1 stratum, but there is one
codimension 2 stratum corresponding to coinciding points; actually it is interesting to find the
complete list of such exceptions). So the most important case to
study is the codimension 1 case.

\vspace{1mm}

{\bf GR2.} A stratification of a subgraph is a substratification
of the original graph (i.e., each stratum for a subgraph is the union of
certain strata for the original graph), hence below we skip the description of
$B_2(G)$ for graphs with 5 vertices other than $K_5$.

\vspace{1mm}

{\bf GR3.} For any stratum there exists a certain subgraph
that {\it locally identifies} the stratum (i.e., for any point $x$
of the stratum there exists a neighborhood $B(x)$ such that any
configuration in $B(x)$ has a nonzero self stress for the
subgraph if and only if this point is on the stratum).

\vspace{1mm}

According to general remarks GR1 and GR2 the most interesting case
is to study the strata of codimension 1 for the complete graph on
$n$ vertices (denoted further by $K_n$). It is possible to find
some of the strata of $K_n$ directly. For the other strata one,
first, should find an appropriate subgraph that locally identifies
the stratum, and then find appropriate surgeries (explained in
Section~3) to reduce the complexity of the subgraph to find
geometric conditions.

{\bf This paper is organized as follows.} In Section~2 we study
the stratification of configuration spaces of tensegrities in the
plane with a small number of vertices. In Subsections~2.1 and~2.2
we briefly describe the trivial cases of two and three point
configurations. Further in Subsections~2.3 and~2.4 we study the
four and the five point cases. In each of the cases we describe
the geometry and the number of strata. In
addition we introduce the adjacency diagram of full dimension
and codimension 1 strata. In Subsections~2.5 and~2.6 we describe
geometric conditions for the codimension 1 strata of 6, 7, and 8
point tensegrities. In Section~3 we present the technique of
surgeries to find geometric descriptions for the strata. In
Subsection~3.1 we describe surgeries that do not change graphs,
and in Subsection~3.2 we show a couple of surgeries in the
two-dimensional case. We introduce a new three-dimensional surgery
in Subsection~3.3. In conclusion, we formulate several open
questions in Subsection~3.4.

\section{Stratification of the space $B_2(K_n)$ for small $n$}

In this section we study the geometry of tensegrity stratifications
for graphs with a small number of vertices. The cases of $n=2,3,4,5$
are studied in full detail. Starting from $n=6$ there are some
gaps in the understanding of tensegrities. Still for $n=6, 7, 8$ the
complete description of the geometric conditions for the strata is
known, we briefly describe several results on them here
(see~\cite{DKS} for more information).

\subsection{Case of two points}

Consider, first, the case of two points ($n=2$). There are only
two graphs on two points: a complete one $K_2$ and a graph without
edges (denote it by $G_{0,2}$).

All the fibers of the base $B_2(G_{0,2})=\r^4$ are of dimension 0,
and, therefore, they are equivalent. Hence the stratification is
trivial.

The complete graph $K_2$ here has only one edge. If two points of
the graph do not coincide then the stress at this edge should be
zero. When two points coincide then the stress at the edge can be
arbitrary, and we have a one-dimensional set of solutions (i.e.,
a fiber). So the base $B_2(K_2)=\r^4$ has a codimension 2 stratum (a
2-dimensional plane). The complement to this stratum is a stratum
of codimension 0.

\subsection{Three point configurations}

There are four different types of graphs here: let $G_{i,3}$ be
the graph with $i$ edges for $i=0,1,2,3$.

In cases $G_{0,3}$ and $G_{1,3}$ the base stratifications are the
following direct products:
$$
B_2(G_{0,3})=B_2(G_{0,2})\times \r^2 \quad \hbox{and} \quad
B_2(G_{1,3})=B_2(K_{2})\times \r^2.
$$
So $B_2(G_{0,3})$ is trivial and $B_2(G_{1,3})$ has a 4-dimensional subspace and its complement as strata.

\vspace{2mm}

The base $B_2(G_{2,3})$ contains five strata. One of them
corresponds to the configuration where three points coincide: the
fiber here is 2-dimensional, this stratum is isometric to $\r^2$.
There are three strata where one of the edges of the graph
vanishes: they are isometric to $\r^4\setminus \r^2$. Finally, the
complement to the union of these strata is the only stratum of
maximal dimension. There are no nonzero tensegrities for a
configuration in this stratum.

\vspace{2mm}

For the complete graph on three vertices
we have, for the first time, codimension 1 strata. There are three
codimension 1 strata, all of them correspond to the following
configuration: three points are in one line. Different strata
correspond to having a different point between the two others.

Let us briefly describe one of such strata. Let $P_i=(x_i,y_i)$ be
the points of the graph ($i=1,2,3$). Then the condition that the three
points are in a line is defined by a quadratic equation:
$$
(x_2-x_1)(y_3-y_1)-(x_3-x_1)(y_2-y_1)=0
$$
This quadric divides the space into two connected components: corresponding to
positively and negatively oriented triangles.

To sum up we present for $B_2(K_3)$ the following table.
\begin{center}
\begin{tabular}{|c|c|c|c|c|c|c|c|}
\hline
Dimension of a stratum &0&1&2&3&4&5&6\\
\hline
Number of such strata  &0&0&1&0&3&3&2\\
\hline
\end{tabular}
\end{center}

\subsection{Stratification of $B_2(K_4)$}

In this subsection we restrict ourselves to the complete graph
$K_4$ (for its subgraphs we apply the reasoning of GR2 above). A plane
configuration of four points in general position admits a unique
tensegrity (up to a multiplicative constant), which is called an
{\it atom}. In~\cite{Guz1} it was proved that any self stress for
$K_n$ is a sum of self-stressed atoms in $K_n$ (i.e., a sum of
certain $K_4\subset K_n$ with scalars). For $K_4$ there are exactly 14
strata of general position.

The strata of codimension 1 correspond to three of four points of
the graph lying in a line. Actually in this case there is no jump of
dimension of the fiber: there is also a unique (up to scalar)
solution corresponding to the three points in a line. But the
stresses on the edges from the fourth point are all zero, and hence
a fiber of this stratum is not equivalent to general
fibers. The number of such strata is 24.

In codimension 2 we have two different types of strata
corresponding to
\begin{itemize}
\item{four points in a line: the dimension of a fiber is 2 (twelve strata);}
\item{two points coincide: the dimension of a fiber is 1 (twelve strata).}
\end{itemize}

In codimension 3 there is one type of strata with configurations
of four points in a line, two of which coincide. Six of them with
the double point in the middle and twelve of them with the double
point not in the middle.

In codimension 4, there are two types of strata:
\begin{itemize}
\item{three points coincide (4 strata);}
\item{two pairs of points coincide (3 strata).}
\end{itemize}

And, finally, there is a codimension 6 stratum when all four
points coincide. We remark that for none of the strata the fiber
is 3-dimensional.

The cardinalities of strata are shown in the following table.
\begin{center}
\begin{tabular}{|c|c|c|c|c|c|c|c|c|c|}
\hline
Dimension of a stratum &0&1&2&3&4&5&6&7&8\\
\hline
Number of strata  &0&0&1&0&7&18&24&24&14\\
\hline
\end{tabular}
\end{center}

\subsubsection{The space of formal configurations}

Let us draw schematically the adjacency of the strata of maximal
dimension via strata of codimension~1.
%Shouldn't it be maximal dimension? --> Indeed!
The dimension of the
stratification space is 8,
% 8? --> Indeed!
let us reduce it to two via factoring
by proper affine transformations. We will use the following
simple proposition.

\begin{prop}\label{aff}
Invertible affine transformations of the plane do not change the equivalence
class of a fiber $W(G,P)$. In other words if $P$ is a
configuration and $T$ an invertible affine transformation of the plane then
$$
W(G,P)\simeq W(G,T(P)).
$$
 \qed
\end{prop}

So instead of studying the stratification itself we restrict to
the set of formal configurations with respect to proper affine
transformations of the plane.

\begin{defn}
We say that a four point configuration $v_1,v_2,v_3,v_4$ is {\it
formal} in one of the following cases:

{\it i$)$} nondegenerate case: {\it a configuration $P_{x,y,+}$
with vertices $v_1=(0,0)$, $v_2=(1,0)$, $v_3=(x,y)$,
$v_4=(x,y{+}1)$ for arbitrary $(x,y)$}.

{\it ii$)$} nondegenerate case: {\it a configuration $P_{x,y,-}$
with $v_1=(0,0)$, $v_2=(1,0)$, $v_3=(x,y)$, $v_4=(x,y{-}1)$ for
arbitrary $(x,y)$}.

{\it iii$)$} degenerate case: {\it a configuration $P_{\Delta,+}$
with $v_1=(0,0)$, $v_2=(1,0)$, $v_3=(0,1)$, $v_4=(\Delta,1)$ for
an arbitrary $\Delta$}.

{\it iv$)$} degenerate case: {\it a configuration $P_{\Delta,-}$
with $v_1=(0,0)$, $v_2=(1,0)$, $v_3=(0,-1)$, $v_4=(\Delta,-1)$ for
an arbitrary $\Delta$}.

{\it v$)$} closure: we add two formal configurations $P_{\pm
\infty}$ with vertices $v_1=(0,0)$, $v_2=(1,0)$, $v_3=(1,0)$,
$v_4=(1,\pm\infty)$.

We denote the set of all formal configurations by $\Lambda_4$.
\end{defn}

In some sense the space $\Lambda_4$ is the space of all
codimension~0 and codimension~1 configurations factored by the
group of proper affine transformations.

\begin{prop}
For any codimension~0 and codimension~1 configuration there exists
a unique formal configuration to which the first configuration can
be affinely deformed. \qed
\end{prop}

The space $\Lambda_4$ is endowed with a natural topology of a quotient space.

\begin{prop}
There is a natural topology of a unit sphere $S^2$ for the set
$\Lambda_4$.
\end{prop}

\begin{proof}
Let us introduce a topology of a unit sphere $S^2$ for
$\Lambda_4$. Consider the configurations of case {\it i}) on
the plane $z=1$: we identify the point $P_{x,y,+}$ with
the point $(x,y,1)$. Consider the projection of this plane to
the upper unit hemisphere $S^2$ from the origin. So we have a one
to one correspondence between the configurations of case {\it i}) and the upper hemisphere.

Similarly we take the plane $z=-1$ for the case {\it ii})
identifying the point $(-x,-y,-1)$ with the configuration
$P_{x,y,-}$ and projecting it to the lower hemisphere.

For the equator of the unit sphere we use all the other cases as
asymptotic directions. First, we associate the configuration
$P_{\Delta,+}$ with the point
$$
(\cos(\pi-\arccotan \Delta),\sin(\pi-\arccotan \Delta),0).
$$
Let us explain the topology at one of such points of the equator.
Suppose we start with $P_{x,y,+}$. The transformation sending the first three points to $(0,0)$, $(1,0)$, and $(0,1)$ is linear with
matrix
$$
\left(
\begin{array}{cc}
1&-x/y\\
0&1/y\\
\end{array}
\right).
$$
Then the image of the fourth point of $P_{x,y,+}$ is
$(-x/y,1{+}1/y)$. While $x$ tends to infinity and $x/y$ tends to
$\Delta$ the last point tends to $(-\Delta,1)$, and hence the
configuration $P_{x,y,+}$ tends to $P_{-\Delta,+}$, as in the
above formula.

Secondly, we associate $P_{\Delta,-}$ with the point
$$
(\cos(-\arccotan \Delta),\sin(-\arccotan \Delta),0)
$$
in a similar way.

Finally, we glue $P_{+\infty}$ and $P_{-\infty}$  to the points
$(1,0,0)$ and $(-1,0,0)$ respectively.
\end{proof}

So, the codimension 0 and 1 stratification of $B_2(K_4)$ can be
derived from the stratification of the sphere. We show the
stereographic projection of $\Lambda_4$ from the point $(0,0,-1)$
to the plane $z = 1$ on Figure~\ref{stratification4D}. There are
four types of strata of codimension~1, they correspond to the fact
that certain three points are in a line. They separate the plane
into 14 connected components. In each of the connected components
we draw a typical type of configuration: $(v_1,v_2,v_3,v_4)$. Here
$v_1$ is blue, $v_2$ is purple, $v_3$ is red and $v_4$ is green.

\begin{figure}
$$\epsfbox{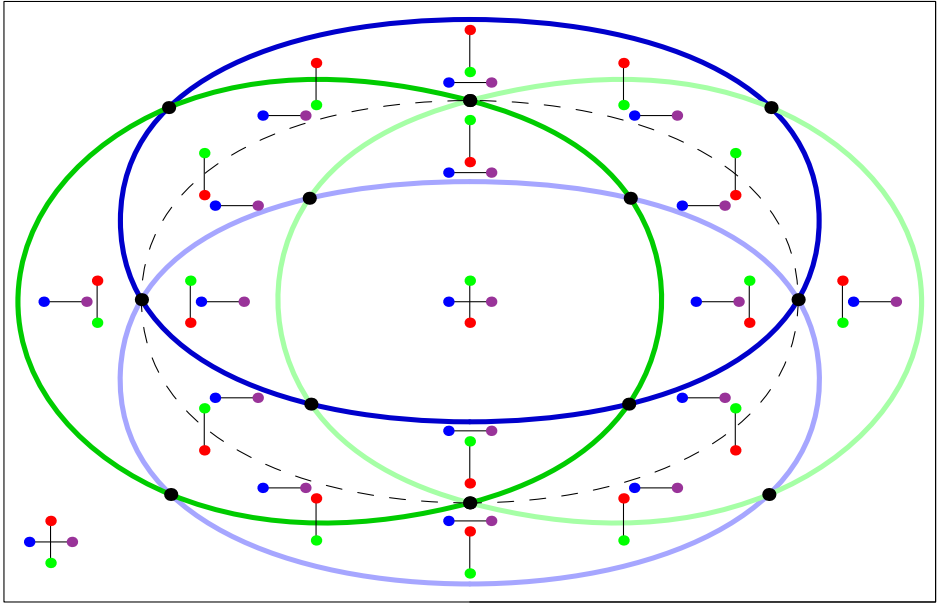}$$
\caption{Stratification of $B_2(K_4)$.}\label{stratification4D}
\end{figure}

\begin{rem}\label{color1} Different geometric conditions are represented
by different colors in the picture, the correspondence is as
follows.
\begin{itemize}
\item
{Light blue strata (6 strata forming a circle) correspond to
configurations with $v_1$, $v_2$, and $v_3$ in a line.}
\item{Dark blue strata (6 strata) contain configurations with $v_1$,
$v_2$, and $v_4$ in a line.}
\item{Light green strata (6 strata) contain configurations with $v_1$,
$v_3$, and $v_4$ in a line.}
\item{Dark green strata (6 strata)
correspond to configurations with $v_2$, $v_3$, and $v_4$ in a
line. We have 24 strata of codimension 1 in total.}

\item{The dashed black line is the projection of the equator. It
corresponds to the degenerate case of parallel segments. The
dashed line is not a stratum, it has the same fiber as all the
points in its neighborhood. While one passes the dashed line the
red-green segment ''rotates'' around the blue-purple segment.}
\end{itemize}
\end{rem}

\subsection{Stratification of $B_2(K_5)$}

\subsubsection{General description of the strata}

We have 264 strata of general position.

As in the two previous cases the strata of codimension 1 correspond to
three points of the graph lying in a line. The number of such
strata is 600.

The following strata are of codimension 2:
\begin{itemize}
\item{twice three points in a line: 270 strata;}
\item{four points in a line: 120 strata;}
\item{two points coincide: 420 strata.}
\end{itemize}

In codimension 3 we have the following cases:
\begin{itemize}
\item{ three points in a line and one double point: 60 strata;}
\item{four points in a line two of which coincide: 180 strata;}
\item{five points in a line: 60 strata.}
\end{itemize}

For codimension 4 we have the following list:
\begin{itemize}
\item{one triple point: 20 strata;}
\item{five points in a line two of which coincide: 120 strata;}
\item{two double points: 30 strata.}
\end{itemize}

In codimension 5 we get:
\begin{itemize}
\item{five points in a line three of which coincide: 30 strata;}
\item{five points in a line with two pairs of points coinciding: 45 strata.}
\end{itemize}

In codimension 6 there are the following strata:
\begin{itemize}
\item{a triple point and a double point: 10 strata;}
\item{one point and one point of multiplicity four: 5 strata.}
\end{itemize}

And, finally, there is a codimension 8 stratum when all five
points coincide.

The cardinalities of the strata are shown in the following table.
\begin{center}
\begin{tabular}{|c|c|c|c|c|c|c|c|c|c|c|c|}
\hline
Dimension of a stratum &0&1&2&3&4&5&6&7&8&9&10\\
\hline
Number of strata  &0&0&1&0&15&75&170&300&810&600&264\\
\hline
\end{tabular}
\end{center}

\subsubsection{Visualization of $B_2(K_5)$}

Let us now describe the structure of the stratification $B_2(K_5)$.
Like in case of $B_2(K_4)$ we introduce a set $\Lambda_5$ which
represents the adjacency of strata of full dimension and of
codimension 1. By definition we put
$$
\Lambda_5=\Lambda_4\times \r^2,
$$
i.e., we consider all the four point configurations of
$\Lambda_4$, and to each configuration we add the fifth point. We
take the product topology for $\Lambda_5$.

So at each point of $\Lambda_4$ we attach an $\r^2$-fiber. It will soon become clear that for any full dimension stratum of $\Lambda_4$ the
corresponding fibration is trivial, but the adjacency is not.

On Figures~\ref{B1K3_1} and~\ref{B1K3_2} we show $\Lambda_5$ in
the following way. We draw the stratification of $\Lambda_4$ and
inside each connected component we show the typical fiber of the
component. The first four points are represented by purple, blue,
green, and red points. The lines passing through any pair of them
divide the fiber into 18 connected components, that correspond to
strata of full dimension. At each such component we write a
letter of the Latin alphabet (we consider capital and small letters as
distinct).

\begin{itemize}
\item{Two regions denoted by the same letter and lying in neighboring
connected components of $\Lambda_4$ separated by light red, dark
red, and black strata are in the same stratum.}

\item{Two regions denoted by the same letter and lying in neighboring
connected components of $\Lambda_4$ separated by light blue, dark
blue, light green, and dark green strata are in distinct strata
which are adjacent to the same codimension 1 stratum.}

\item{Two regions denoted by a distinct letter and lying in neighboring
connected components of $\Lambda_4$ are not in one stratum and are not
adjacent to the same codimension 1 stratum.}
\end{itemize}

The light blue, dark blue, light green, and dark green strata
represent the same geometric conditions as in Remark~\ref{color1}
above. For the remaining strata we have:
\begin{itemize}
\item{The dark red stratum symbolizes that the line through the red and blue points
is parallel to the line through the green and purple points.}
\item{The light red stratum symbolizes that the line through the red and purple points
is parallel to the line through the green and blue points.}
\item{The black stratum symbolizes that the line through the red and green points
is parallel to the line through the purple and blue points.}
\end{itemize}

\begin{rem}
The configuration space $B_2(K_5)$ has several obvious symmetries.
First, there is the group of permutations $S_5$ that acts on the
points of $B_2(K_5)$; these symmetries are hardly seen from
Figures~\ref{B1K3_1} and~\ref{B1K3_2} since the representation is
not $S_5$-symmetric. Secondly, there is a symmetry about the
origin that sends configurations from $B_2(K_5)$ to themselves, on
Figures~\ref{B1K3_1} and~\ref{B1K3_2} we used capital and small
letters to indicate this symmetry (for instance, the strata of
''a'' contain centrally symmetric configurations to the
configurations of the strata ''A'').
\end{rem}

\begin{figure}
$$\epsfbox{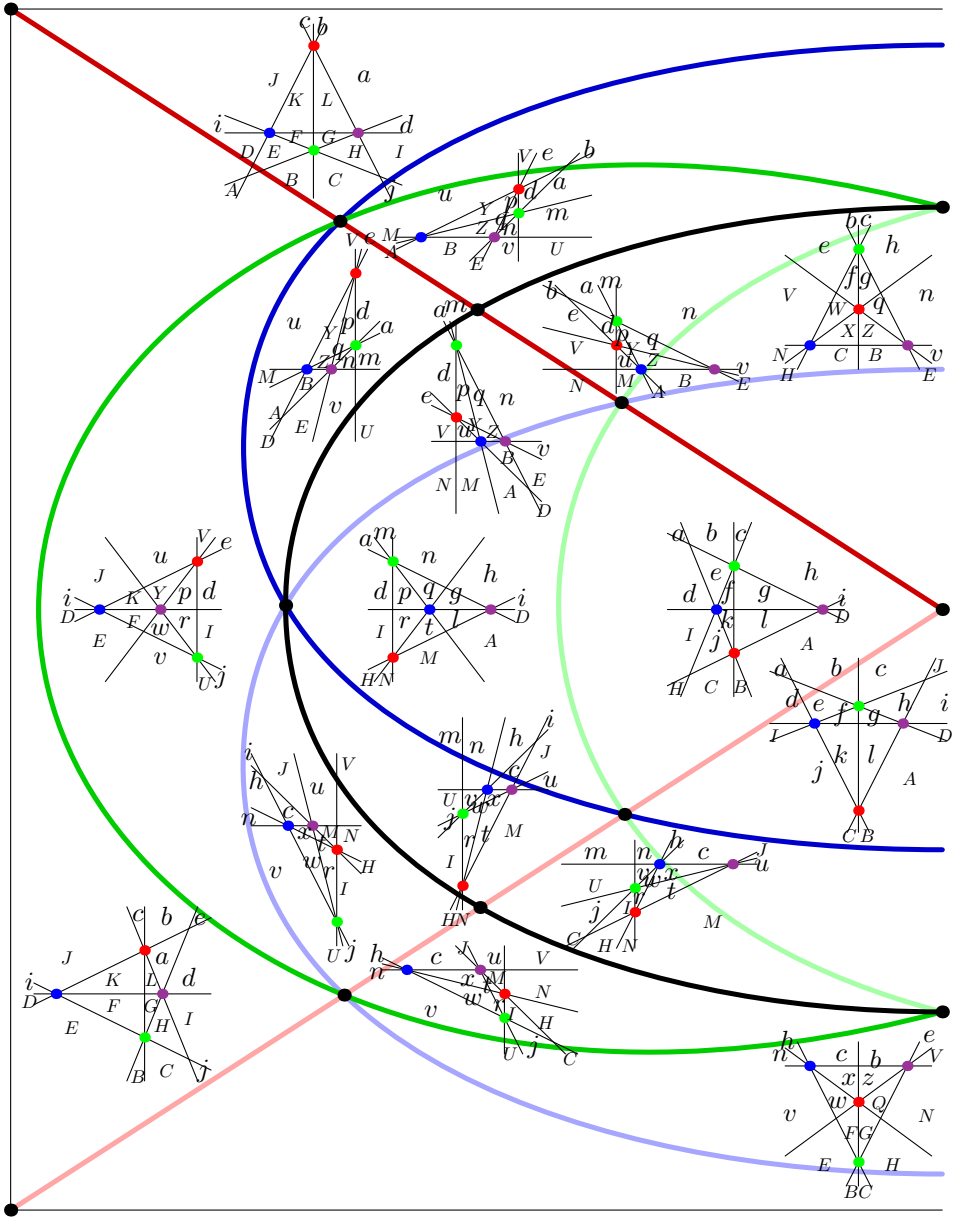}$$
\caption{Stratification of $B_2(K_5)$  (Left part).}\label{B1K3_1}
\end{figure}

\begin{figure}
$$\epsfbox{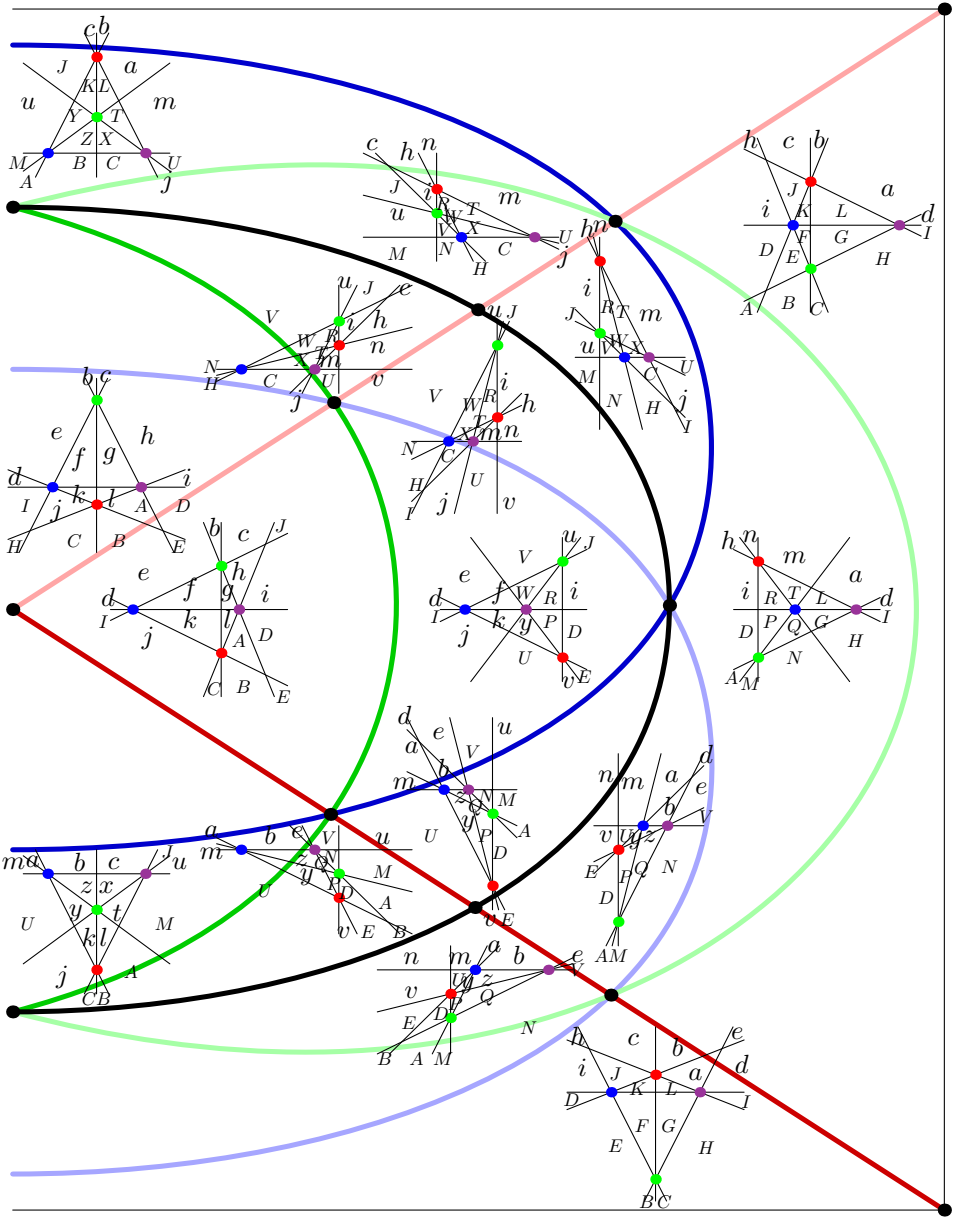}$$
\caption{Stratification of $B_2(K_5)$  (Right
part).}\label{B1K3_2}
\end{figure}

As in the case of 4 point configurations we skip the subgraphs of
$K_5$, see the second general remark above (GR2).

\subsection{Essentially new strata in $B_2(K_6)$}

The stratification of $B_2(K_6)$ is much more complicated, at this
moment we do not even know how many strata of distinct dimension are
present in the stratification.

According to GR1 the first step in studying the stratification of
$B_2(K_6)$ is to study all possible distinct types of strata of
codimension~1. In the examples of $K_n$ for $n<6$ we only have strata corresponding to the following geometric condition: three points
are in a line. For the case of 6 points we get two additional
types of strata: six points on a conic, and three lines passing through
three pairs of points have a unique point of intersection.

So the following are three codimension~1 strata (appeared
in~\cite{WW} by N.~L.~White and W.~Whiteley):
\begin{itemize}
\item{three points in a line;}

\item{the lines $v_1v_2$, $v_3v_4$, and $v_5v_6$ meet in one point (or all parallel);}

\item{all the six points are on a conic.}
\end{itemize}

We conclude this subsection with the following problems.
\begin{prob}
Find a description of $B_2(K_6), B_3(K_4)$ and $B_3(K_5)$ similar to the ones for $B_2(K_4)$ and
$B_2(K_5)$ shown in the previous subsections.
\end{prob}
%$B_3(K_4)$ and $B_3(K_5)$ should be fairly easy to describe. Add as problem? --> I added this.
\subsection{A few words about the case $n>6$}

In~\cite{DKS} we have studied strata of the 7 and 8 point
configurations. There are 4 distinct types of codimension 1 strata for 7 points
and 17 types for 8 points.

The 4 types of codimension 1 strata for 7 points are defined by
the following geometric conditions:
\begin{itemize}
\item{three points in a line;}

\item{the lines $v_1v_2$, $v_3v_4$, and $v_5v_6$ meet in one point (or all parallel);}

\item{the lines $v_1v_2$, $v_3v_4$, and $v_5p$ (where
$p$ is the intersection of the lines $v_2v_6$ and $v_3v_7$) have a
common nonempty intersection;}

\item{the six points $v_1$, $v_2$, $v_3$, $v_4$, $v_5$,
and $p$ (where $p$ is the intersection of the lines $v_1v_6$ and
$v_3v_7$) are on a conic.}
\end{itemize}

For the list of strata of 8 point configurations we refer
to~\cite{DKS}.

It turns out that the geometric conditions of any codimension 1
stratum can be obtained by the following procedure. Consider the
points of configuration $P$; for each two pairs of points
$(v_i,v_j)$ and $(v_k,v_l)$ of this configuration consider the
point of intersection of the lines $v_iv_j$ and $v_kv_l$. This
leads to a bigger configuration of points including $P$ and the above
intersections, we denote it by $U(P)$. This operation can be
iteratively applied infinitely many times, which results in a {\it
universal set}
$$
U^\infty(P)=\bigcup\limits_{m=0}^{\infty} U^m(P).
$$
Any condition for a codimension~1 stratum is always as follows:
{\it three certain points of $U^\infty(P)$ are in a line} (for the
details, see for instance~\cite{Rot1} and~\cite{DKS}).

\begin{exmp}
The condition {\it the lines $v_1v_2$, $v_3v_4$, and $v_5v_6$ meet
in one point} in terms of points of $U^1(P)=U(P)$ is as follows.
{\it The points $v_1$, $v_2$, and $p=v_3v_4\cap v_5v_6$ are in a
line}.
\end{exmp}

\begin{rem}
For simplicity reasons we omit discussions of cases where certain
lines $v_iv_j$ and $v_kv_l$ are parallel, due to the fact that
this situation is never generic for codimension 1 strata. In
general one may think that if the lines $v_iv_j$ and $v_kv_l$ are
parallel, then their intersection point is in the line at infinity
in the projectivization of $\r^2$.
\end{rem}

\begin{rem}\label{conic}
At first glance, the condition {\it six points are on a conic} is
of different nature. Nevertheless, it is a relation on the points
of the configuration in $U^1(P)$ described by Pascal's theorem:
{\it The intersections of the extended opposite sides of a hexagon
inscribed in a conic lie on the Pascal line.} See also
Example~\ref{conic2} below.
\end{rem}

\begin{prob}
Describe all the possible different types of strata for 9 points.
\end{prob}

\begin{prob}
How to calculate the number of different types of strata for $n$
points with arbitrary $n$?
\end{prob}

It is also interesting to have an answer for the following question:
{\it how many iterations $($i.e., find the minimal $m$ for
${U^m(P)}$$)$ does one need to perform to describe all conditions for
the codimension 1 strata of $n$-point configurations $P$?}
%Is it clear that this number is finite? --> Good question; I have no idea...
\begin{prob}
Which configurations of $U^m(P)$ define the same geometric
condition?
\end{prob}
This problem is a kind of question of finding generators and
relations for the set of all conditions. Let us show one type of
such ''relations'' in the following example.

\begin{exmp}\label{conic2}
Consider the condition: six points $v_1,v_2,\ldots,v_6$ are on a
conic. This condition is described by configurations contained in
$U^1(P)$ via Pascal's theorem:
\begin{center}
The points $p,q,r$ are in a line for \quad $
\left\{\begin{array}{l}
p=v_{\sigma(1)}v_{\sigma(2)}\cap v_{\sigma(4)}v_{\sigma(5)}\\
q=v_{\sigma(2)}v_{\sigma(3)}\cap v_{\sigma(5)}v_{\sigma(6)}\\
r=v_{\sigma(3)}v_{\sigma(4)}\cap v_{\sigma(6)}v_{\sigma(1)}\\
\end{array}
\right.,
$
\end{center}
where $\sigma$ is an arbitrary permutation of the set of six
elements. So, there are 60 different configurations of $U^1(P)$
defining the same geometric condition.
\end{exmp}

\section{Further study of strata: surgeries}
We now look into subgraphs contained in a particular stratum and ask the
basic question on the dimension of the fiber.

Even graphs of very low connectivity admit non-zero tensegrities, for
disconnected or one-connected graphs we may simply examine the
connected or 2-connected components.
Also 2-connected graphs may be decomposed via the 2-sum, see~\cite{2-sum}: Consider
graphs $G_1$ and $G_2$, their configurations $P_1$ and $P_2$ admitting tensegrities with
$p_1 q_1$ a cable in $G_1(P_1)$ and $p_2 q_2$ a strut in $G_2(P_2)$. We form the {\it 2-sum}
$G_1 \bigoplus G_2$ by identifying $p_1$ with $p_2$ and $q_1$ with $q_2$ and removing the
identified edge. We can inherit a configuration $P$ from $P_1$ and $P_2$ by fixing $P_1$ and
properly dilating, rotating and translating $P_2$. It is clear that
$$
\dim W(G_1 \bigoplus G_2 ,P) =\dim W(G_1 ,P_1 ) + \dim W(G_2 , P_2 ) -1.
$$

Since 2-sum decomposition is canonical, we can describe geometric conditions
for 2-connected graphs by geometric conditions on their 3-blocks. For example
the geometric condition for $G$ in Figure~\ref{2sumfig} is that the lines
$v_1v_2$, $v_3v_4$, and $v_5v_6$ meet in one point.

\begin{figure}[htb]
\centering
\epsfbox{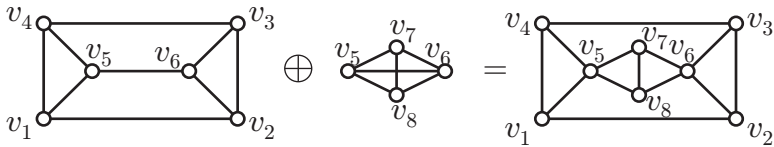}
\caption{The 2-sum of a triangular prism with $K_4$\label{2sumfig}}
\end{figure}

\subsection{Subgraphs related to codimension 1 strata}

As we have already mentioned in GR3, for any codimension 1 stratum
there exists at least one subgraph of $K_n$ that generically does
not admit tensegrities but at this stratum admits a
one-dimensional family of tensegrities. Let us show such subgraphs
for the codimension one strata of $B_2(K_6)$ and $B_2(K_7)$.

\begin{exmp}
In the case of $K_6$ we have three strata of
different geometrical nature. The first triangular subgraph
(Figure~\ref{strata6}, left) is related to the strata with three
points in a line. The second
(Figure~\ref{strata6}, middle) corresponds to the strata  whose
three pairs of points generate lines passing through one point.
The last one (Figure~\ref{strata6}, right) corresponds to the
configurations of six points on a conic.

\begin{figure}
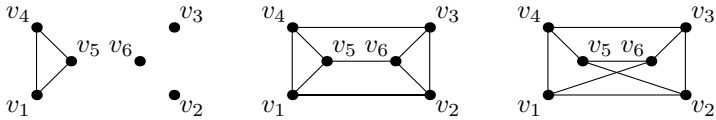

$$\epsfbox{table.200} \qquad \epsfbox{table.1} \qquad \epsfbox{table.2}$$
\caption{Examples of subgraphs of $K_6$ admitting tensegrities at
codimension 1 strata of $B_2(K_6)$.}\label{strata6}
\end{figure}
\end{exmp}

\begin{exmp}
In the case of $K_7$ there are the following new examples of
subgraphs, corresponding to the main 4 different types of strata.

From the left to the right we have the following geometric
conditions
\begin{itemize}
\item{$v_1$, $v_2$, and $v_3$  are in a line;}

\item{the lines $v_1v_2$, $v_3v_4$, and $v_5v_6$ meet in one point;}

\item{the lines $v_1v_2$, $v_3v_4$, and $v_5p$ (where
$p=v_2v_6 \cap v_3v_7$) have a common point;}

\item{the six points $v_1$, $v_2$, $v_3$, $v_4$, $v_5$,
and $p$ (where $p=v_1v_6\cap v_3v_7$) are on a conic.}
\end{itemize}

\begin{figure}
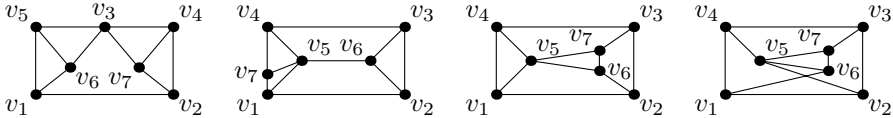

$$\epsfbox{table.3} \quad \epsfbox{table.4} \quad \epsfbox{table.5} \quad \epsfbox{table.6}$$
\caption{Examples of subgraphs of $K_7$ admitting tensegrities at
codimension 1 strata of $B_2(K_7)$.}\label{strata7}
\end{figure}
\end{exmp}

Note that the example for three points in a line is actually the 2-sum of
a triangle with two atoms, so the only way for a non-zero
self stress on the edges is to have $v_1$, $v_2$, and $v_3$, the vertices of the triangle, in a line.

\begin{rem}
Geometric conditions for the graphs with 8 and fewer vertices are
given in~\cite{DKS}. Several of those geometric
conditions were described before in terms of bracket polynomials
in~\cite{WW} by N.~L.~White and W.~Whiteley. We also refer to the
paper~\cite{BC} by E.~D.~Bolker and H.~Crapo for the relation of
bipartite graphs with rectangular bar constructions.
\end{rem}

\subsection{Surgeries on subgraphs that change geometric conditions in a predictable way}

In this subsection we present several surgeries that allow to guess
the geometric conditions for new strata (characterized by certain
subgraphs) via other strata (characterized by these graphs modified in a
certain way). We call such modifications of graphs {\it
surgeries}.

\subsubsection{Surgeries that do not change geometric conditions}

Let $G$ be a graph, denote by $G_e$ the graph with an edge $e$
removed.

\begin{prop}\label{Proposition 1.9.}{\bf(Edge exchange)}
Consider a graph $G$ and a subgraph $H$, and let $e_1$ and $e_2$
be two edges of $H$. Let $P$ be a configuration for which $\dim
W(H,P)=1$. Suppose also that the self stresses of $H$ do not
vanish at the edges $e_1$ and $e_2$. Then we have
$$
\dim W(G_{e_1},P) =\dim W(G_{e_2},P).
$$
\qed
\end{prop}
In the situation of Proposition~\ref{Proposition 1.9.} the strata of $G_{e_1}(P)$ and $G_{e_2}(P)$ are defined by the same
geometrical conditions.

\subsubsection{Two two-dimensional surgeries that change geometric conditions}

The first surgery is described in the following proposition.

\begin{prop}\label{operation1}
Consider the frameworks $G(P)$, $G_1^{I}(P_1^{I})$, and
$G_2^{I}(P_2^{I})$ as on the figure:
$$
\epsfbox{operations.1}
$$
If none of the triples of points $(p,v_2,v_3)$, $(q,v_2,v_3)$,
$(p,v_2,v_4)$, $(q,v_3,v_4)$ and $(v_2,v_3,v_4)$ are
%not
on a line
then we have
$$
\dim W(G_1^{I},P_1^{I}) = \dim W(G_2^{I},P_2^{I}).
$$
\end{prop}

\begin{exmp}
Let us consider a simple example of how to get a geometric condition
for the graph
$$
\epsfbox{table.5}
$$
to admit a tensegrity knowing all geometric conditions for 6-point
graphs. Let us apply Surgery I to the points $v_5$, $v_6$, $v_7$.
We have:
$$
\epsfbox{table.5} \epsfbox{table.202}.
$$
The geometric condition to admit a tensegrity for the graph on the right is:

{\it the lines $v_1v_2$, $v_3v_4$ and $v_5p$ intersect in a
point}.
\\
Hence the geometric condition for the original graph is:

{\it the lines $v_1v_2$, $v_3v_4$ and $v_5p$ intersect in a point,
where $p=v_2v_6\cap v_3v_7$}.
\end{exmp}

Now let us show the second surgery.

\begin{prop}\label{operation2}
Consider the frameworks $G(P)$, $G_1^{II}(P_1^{II})$, and
$G_2^{II}(P_2^{II})$ as on the following figure:
$$\epsfbox{operations.2}$$
If none of the triples of points $(p,q,v_1)$, $(p,v_1,v_4)$,
$(r,v_1,v_4)$, $(q,v_1,v_4)$, $(s,v_1,v_4)$, or $(r,s,v_4)$ lie on
a line then we have
$$
\dim W(G_1^{II},P_1^{II}) = \dim W(G_2^{II},P_2^{II}).
$$
\qed
\end{prop}

\begin{rem}
Both surgeries were shown in~\cite{DKS}. There is a
certain analogy of the first surgery to $\Delta Y$ exchange in
matroid theory (see for instance~\cite{Whi}
and~\cite{JJ} for the connections between matroids and rigidity theory), but it is not exactly the same.
\end{rem}

\begin{rem}
Actually these surgeries are valid in the multidimensional case as
well under the condition that certain points are in one plane.
\end{rem}

\subsection{A new tensegrity surgery in $\r^3$}

We conclude this paper with a single surgery for tensegrities in
$\r^3$.

\begin{prop}\label{3doperation}
Consider a graph $G$ and frameworks $G(P)$, $G_1(P_1)$, and
$G_2(P_2)$ as follows:
$$
\epsfbox{3doperations.1}
$$
Denote the plane $v_2v_3v_4$ by $\pi_1$.  Suppose that the couples of edges
$e_1$ and $e_2$, $e_3$ and $e_4$, $e_5$ and $e_6$ define planes $\pi_2$, $\pi_3$, and $\pi_4$, different from $\pi_1$. Assume that
$\pi_2\cap\pi_3\cap\pi_4$ is a one point intersection.

If $G_1(P_1)$ and $G_2(P_2)$ have nonzero stress on the edges connecting $v_1$,
$v_2$, $v_3$, and $v_4$ then
$$
\pi_1\cap\pi_2\cap\pi_3\cap\pi_4=v_1.
$$
In this case we additionally have
$$
\dim W(G_1,P_1) = \dim W(G_2,P_2).
$$
\end{prop}

\begin{proof}
The first statement follows since $v_1$ only has valency 3 in $G_2(P_2)$, so $v_1$, $v_2$, $v_3$, and $v_4$ need to be coplanar to have a nonzero edge stress.
Now we explain how to map $W(G_1,P_1)$ to $W(G_2,P_2)$. The inverse
map is simply given by the reverse construction. By the conditions
$v_1$ is the intersection point of the planes $\pi_1$, $\pi_2$, and
$\pi_3$. We add the uniquely defined plane atom on $v_1,v_2,v_3,v_4$ to
$G_1(P_1)$ that cancels the edge stress on $v_2v_3$. Since the
plane $\pi_1$ does not coincide with the plane $\pi_2$ spanned by the forces on $e_1$ and $e_2$,
the edge stress on $v_2v_4$ is also canceled. By the same reasons
the edge stress on $v_3v_4$ is canceled as well. This uniquely
defines a self stress on $G_2(P_2)$.
\end{proof}

\subsection{Some related open problems}

The next goal in this approach is to continue to study the
geometry of the strata. Ideally {\it one would like to find techniques
that will give geometric conditions for a graph via its
combinatorics}. This question seems to be a very hard open problem.
The study of surgeries is the first step to solve it at least in
codimension 1.

For a start we propose the following open question.
\begin{prob}
Find all geometric conditions for the strata of 9 point
tensegrities.
\end{prob}
The surgeries introduced in this section were extremely useful for
the study of 8 point configurations (see in~\cite{DKS}). We think
that it is not enough to know only these surgeries to find all
the geometric conditions. This gives rise to another question.
\begin{prob}
Find other surgeries on graphs that predictably change the
geometric conditions.
\end{prob}

As far as we know there is no  systematic study of strata for
tensegrities in $\r^3$ or higher dimensions: these cases are
much more complicated than the planar case.
At least the stratification of $B_3(K_5)$ should have a description
similar to that of $B_2(K_4)$, since 5 points in general position
in $\r^3$ admit a unique non-zero self stress.

Additionally one should examine the rigidity properties of subgraphs
in a stratum. For $K_4$ we have 14 strata of full dimension. For
8 of them the convex hull is a triangle, in 5 of the strata the points are in convex position.
A tensegrity for the convex position has 4 struts (cables) and two cables (struts), while in
the non-convex case there are three cables and three struts. All of these tensegrities are
(infinitesimally) rigid and struts and cables may be exchanged without destroying rigidity.
However, when viewed as graphs embedded in $\r^3$ only half of them are rigid. For the convex case, there must be cables on the convex hull and two struts. In the non-convex case there must be a triangle of struts on
the convex hull and three cables in the interior, termed a spider web by R.~Connelly. None of
these are proper forms in the sense of B.~Gr\"{u}nbaum. They are minimally rigid, but in the convex case they
have members intersecting in a vertex other than a vertex of the graph, in the non-convex case there
is a vertex without a strut. B.~Gr\"{u}nbaum in his lectures on lost mathematics~\cite{Grunbaum} asks about the number
of proper forms given $n$ struts. On 3 struts, there is only one tensegrity which is minimally rigid with
edges only intersecting at vertices and such that every vertex is endpoint of at least one strut. For 4 struts there are at least 20 forms, but it is not known how many there are. The number of forms on $n$ struts is bounded by the number of strata
on $B_3(K_n)$. For the hierarchies of the various kinds of rigidity see~\cite{Con}.

\vspace{5mm}

\end{document}